\documentclass{amsart}
\usepackage{amsfonts,amssymb,amscd,amsmath,enumerate,verbatim,calc}
\usepackage[all]{xy}

\newcommand{\CM}{Cohen-Macaulay}

\newcommand{\wrt}{with respect to}

\newcommand{\I}{\mathbb{I} }
\newcommand{\n}{\mathfrak{n} }
\newcommand{\m}{\mathfrak{m} }

\newcommand{\C}{\mathcal{C} }

\newcommand{\G}{\mathbb{G} }
\newcommand{\rt}{\rightarrow}

\newcommand{\wh}{\widehat }

\newcommand{\image}{\operatorname{image}}
\newcommand{\rank}{\operatorname{rank}}

\newcommand{\Min}{\operatorname{Min}}

\newcommand{\depth}{\operatorname{depth}}
\newcommand{\type}{\operatorname{type}}

\newcommand{\injdim}{\operatorname{injdim}}

\newcommand{\Ass}{\operatorname{Ass}}
\newcommand{\Assh}{\operatorname{Assh}}

\newcommand{\Syz}{\operatorname{Syz}}

\newcommand{\height}{\operatorname{height}}
\newcommand{\Hom}{\operatorname{Hom}}
\newcommand{\Ext}{\operatorname{Ext}}

\theoremstyle{plain}

\newtheorem{theorem}{Theorem}[section]
\newtheorem{corollary}[theorem]{Corollary}
\newtheorem{lemma}[theorem]{Lemma}
\newtheorem{proposition}[theorem]{Proposition}

\newtheorem{conjecture}[theorem]{Conjecture}

\theoremstyle{definition}

\newtheorem{s}[theorem]{}

\newtheorem{example}[theorem]{Example}

\newtheorem{construction}[theorem]{Construction}

\theoremstyle{remark}

\begin{document}

\title{On generalization of two results of Foxby}
\author{Tony~J.~Puthenpurakal}
\date{\today}
\address{Department of Mathematics, Indian Institute of Technology Bombay, Powai, Mumbai 400 076, India}

\email{tputhen@gmail.com}

 \begin{abstract}
Let $(A,\m)$ be a Noetherian local ring of dimension $d$ and let $M$ be a finitely generated $A$-module. Assume $M$ has rank $r > 0$. We show that if $M$ is NOT \CM \ then $\mu_d(\m, M) > r$. If further $A$ is unmixed  and $\mu_n(\m, M) \leq 1$ for some $n \geq d$ then we prove $\injdim M < \infty$ and $A$ is \CM.
\end{abstract}
 \maketitle
\section{introduction}
Our paper is motivated by Foxby's  beautiful paper, \cite{F}.
Let $(A,\m)$ be a Noetherian local ring of dimension $d$ and let $M$ be finitely generated $A$-module. Let $P$ be a prime ideal of $A$ and let $\kappa(P)$ be the residue field of $A_P$. The number
$\mu_i(P, M) = \dim_{\kappa(P)} \Ext^i_{A_P}(\kappa(P), M_P)$ is called the $i^{th}$ Bass number of $M$ \wrt \  $P$.

If $A$ is Gorenstein then $\mu_d(\m, A) = 1$. Vasconcelos conjectured the converse, i.e., if $\mu_d(\m, A) = 1$ then $A$ is Gorenstein. This conjecture was proved by Foxby when $A$ contains a field (see  \cite[3.7]{F} and \cite[9.6.3]{BH}) and by Roberts in general, see \cite{R}. If $A$ is a \CM \ and if $M$ is a faithful maximal \CM  \ $A$-module with $\mu_d(\m, M) =1$ then $M$ is the canonical module of $A$, so in particular $\injdim_A M = d$.

 Foxby also essentially \cite[Conjecture B]{F} made the following:
\begin{conjecture}
  \label{my-conj}
  Let $(A,\m)$ be a Noetherian local ring of dimension $d$ and let $M$ be a finitely generated $A$-module. If $M$ is a faithful $A$-module and if $\mu_d(\m, M) = 1$ then $A$ is \CM \ and $M$ is a canonical module of $A$.
\end{conjecture}
Foxby showed the conjecture is true when $A$ contains a field or when $A$ is unmixed or when $\depth M \geq \dim A - 2$. We give a considerably simpler general proof of Conjecture \ref{my-conj}. Techniques used for our proof also proves the following result:
\begin{theorem}
\label{rank-non-cm}
 Let $(A,\m)$ be a Noetherian local ring of dimension $d$ and let $M$ be a finitely generated $A$-module. Assume $M$ has a positive rank $r$.
 Then $\mu_d(\m, M) \geq r$.  Furthermore the following assertions are equivalent:
 \begin{enumerate}[\rm (i)]
   \item $\mu_d(\m, M) = r$
   \item $A$ is \CM \ and $A_P$ is Gorenstein for all $P \in \Min(A)$. Furthermore $M$ is maximal \CM, $\injdim M < \infty$  and $\wh{M} \cong \omega_{\wh{A}}^r$.
 \end{enumerate}
\end{theorem}
In the above theorem $\omega_{\wh{A}}$ denotes the canonical module of the completion $\wh{A}$ of $A$.

Recall a Noetherian local ring $A$ is said to be unmixed if $\dim \wh{A}/Q = \dim \wh{A}$ for all associate primes  $Q$ of the completion $\wh{A}$ of $A$.
Foxby also proved the following result, see \cite[4.3]{F}:
\begin{theorem}\label{fm}
Let $(A,\m)$ be a Noetherian local ring of dimension $d$ and assume $\Min \wh{A} = \Ass \wh{A}$. If $\mu_n(\m, A) \leq 1$ for some $n \geq d$ then $A$ is a Gorenstein ring.
\end{theorem}

In view of Conjecture \ref{my-conj} we might ask whether there is a version available for modules. We prove
\begin{theorem}
\label{my-thm}
Let $(A,\m)$ be an unmixed Noetherian local ring of dimension $d$ and let $M$ be a finitely generated $A$-module. Assume that $\rank M = r \geq 1$. If $\mu_n(\m, M) \leq 1 $ for some  $n \geq d$ then $\injdim M < \infty$, $A$ is \CM \ and $A_P$ is Gorenstein for all $P \in \Min(A)$.
\end{theorem}
We give an example (see \ref{alive}) which shows that the result does not hold if we do not assume that  $M$ has a rank.
It might be tempting that under the hypotheses of the theorem if $r = 1$ then $M \cong \omega_A$, the canonical module of $A$. We give a counter-example to this, see \ref{sty}. However we can prove
\begin{proposition}
  \label{nuv} (under the assumptions of \ref{my-thm})  Also assume $A$ is a quotient of a Gorenstein local ring. Then
  $\nu(M) \geq r\type(A)$ (here $\nu(M)$ is the number of minimal   generators of $M$).
  Furthermore
   the following assertions are equivalent:
  \begin{enumerate}[\rm (i)]
    \item  $M \cong \omega_A^r$
    \item $\nu(M) = r\type(A)$.
  \end{enumerate}
\end{proposition}

Next we  state our result regarding partial Euler characteristics of Bass numbers.
We show
\begin{theorem}
\label{bass}
Let $(A,\m)$ be a  Noetherian local ring of dimension $d$ and let $M \neq 0$ be a finitely generated $A$-module. Then $\sum_{i = 0}^{d-1}(-1)^i\mu_{d-1- i}(\m, M)  \geq 0$.
The following assertions are equivalent:
\begin{enumerate}[\rm (i)]
  \item $\sum_{i = 0}^{d-1}(-1)^i\mu_{d-1- i}(\m, M)  = 0$.
  \item $\mu_i(\m, M) = 0$ for $i < d$.
  \item $M$ is a maximal \CM \ $A$-module.
\end{enumerate}
\end{theorem}
We need the new intersection theorem to prove the above result. Theorem \ref{bass} also implies the Bass conjecture, see \ref{bc}.

\emph{Technique used to prove our results:}

 We need the following result due to Foxby \cite[5.1]{F}.
\begin{lemma}
\label{bella} Let $(A,\m)$ be a Noetherian local ring of dimension $d \geq 1$ and let $M$ be a finitely generated $A$-module. Let $P$ be a prime ideal of with  $\dim A/P = 1$. Then for all $i \geq 0$ we have
$$\mu_i(P, M) \leq \mu_{i+1}(\m, M).$$
\end{lemma}
Foxby uses dualizing complexes for the proof of the Lemma. We give an elementary
proof of it.

We also need the (slight) generalization of the new intersection theorem by Foxby, see \cite[1.2]{F}.
\begin{theorem}\label{int}
Let $(A,\m)$ be a Noetherian local ring. Let $0 \rt F_s \rt F_{s-1} \rt \cdots \rt F_0 \rt 0$ be a non-trivial complex of finitely generated complex of finitely generated free $A$-modules and let $t \geq 0$ be an integer such that $\dim H_i(F_\bullet) \leq i + t$ for all $i \geq 0$. Then $\dim A \leq s + t$.
\end{theorem}
Foxby proves this result for local ring containing a field. His proof uses Hochster's theorem on  existence of big \CM \ modules over such rings. As existence of big \CM \ modules have been shown to exist for all local rings it is easy to check that  Foxby's proof of Theorem \ref{int} for equicharacteristic local rings generalize to all local rings.

We now describe in brief the contents of this paper. In section two we discuss some preliminaries. In section three we give a proof of Lemma \ref{bella}.
In section four we prove Theorem \ref{bass} and give our proof of Bass's conjecture.
 In the next section  we give our proof of Conjecture \ref{my-conj}. In the  section six we give a proof of Theorem \ref{rank-non-cm}. In section seven we prove Theorem \ref{my-thm} and Proposition \ref{nuv}.

\section{Preliminaries}
In this section we discuss a few preliminary facts that we need.
Throughout $(A,\m)$ is a Noetherian local ring of dimension $d$ and $M$ is a finitely generated $A$-module. By $\Min(M)$ (respectively $\Ass(M)$) we denote the set of minimal primes (respectively associate primes) of $M$. By $\Assh(A)$ we denote the set of primes $P$ of $A$ with $\dim A/P = d$.
\s\label{faithful} If $M$ is a faithful $A$ module then it is easy to check that $M_P$ is a faithful $A_P$ module for all primes $P$ of $A$. Similarly the completion $\wh{M}$ is a faithful $\wh{A}$-module. Let $M$ be generated by $\{ u_1, \ldots, u_l \}$. Consider the homomorphism $\psi \colon A \rt M^l$ defined by $\psi(a) = (au_1, \ldots, au_l)$. If $M$ is faithful then $\psi $ is injective. In particular $\Ass(A) \subseteq \Ass(M)$.

\s \label{zero} Conjecture \ref{my-conj} holds when $\dim A = 0$; see \cite[3.3.13]{BH}.

\s \label{mod-mult} Let $I$ be an $\m$-primary ideal of $A$. Let $M$ be a finite $A$-module of dimension $r$. Then $e(I, M) = \lim_{n\rt \infty} r!\ell(M/I^nM)/n^r$ is the multiplicity of $M$ \wrt \ $I$.
Let $\C_q(A)$ denote the abelian category of all finitely generated $A$-modules of dimension $\leq q$. Define the modified multiplicity on $\C_q(A)$ as
\[
e_q(I, M) = \begin{cases}
              e(I,M), & \mbox{if } \ \dim M = q. \\
              0, & \mbox{if } \ \dim M < q.
            \end{cases}
\]
The modified multiplicity is additive  under short exact sequences, i.e., if $0 \rt M_1 \rt M_2 \rt M_3 \rt 0$ is a short exact sequence in $\C_q(A)$ then by \cite[4.7.7]{BH}
$$e_q(I, M_2) = e_q(I, M_1) + e_q(I, M_3).$$
 We also have
 \[
 e_q(I, M) = \sum_{P}\ell_{A_P}(M_P) e_q(I, A/P),
 \]
 where the sum is over all prime ideals $P$ of $A$ with $\dim A/P = q$; see \cite[4.7.8]{BH}.
In particular
\begin{equation*}
  e_d(I, M) = \sum_{P \in \Assh(A)}\ell_{A_P}(M_P) e_d(I, A/P). \tag{*}
\end{equation*}

 We will need the following computation of multiplicites.
 \begin{theorem}
 \label{rachel} Let $(A,\m)$ be a complete local ring of dimension $d \geq 1$. Assume $A$ is a quotient of a complete Gorenstein local ring $R$ with $\dim R = d$. Let $M$ be a finitely generated  $A$-module. Assume there exists $r \geq1$ such that for every prime $P \in \Assh(A)$ we have  $\ell_{A_P}(M_P) = r \ell_{A_P}(A_P)$.
 Set  $M^\dagger = \Hom_R(M, R)$. Then we have
 \begin{enumerate}[\rm (1)]
   \item $e_d(\m, M) = r e_d(\m, A)$.
   \item $e_d(\m, M^\dagger) = re_d(\m, A)$.
 \end{enumerate}
 \end{theorem}
 \begin{proof}
 (1) The assertion holds by (*).

  (2) Let $P \in \Assh(A)$ and let $Q$ be inverse image of $P$ in $R$. We note that $\height Q = 0$ and there exists a surjective homomorphism of Artin local rings $R_Q \rt A_P$. We have
  $$ \ell(M^\dagger)_P = \ell_{R_Q}(\Hom_R(M_P, R_Q)) = \ell_{A_P}(M_P) = r \ell_{A_P}(A_P).$$
  Here the second equality holds by \cite[3.2.12]{BH}. Now the result follows by (*).
 \end{proof}
\section{Proof of Lemma \ref{bella}}
In this section we give
\begin{proof}[Proof of Lemma \ref{bella}]
We have nothing to show if $\mu_i(P.M) = 0$. Also note that if $\mu_{i+1}(\m, M) = 0$ then $\Ext^{i}_A(A/P, M) = 0$  (see \cite[3.1.13]{BH})  and so $\mu_i(P, M) = 0$.
Thus we may assume that $\mu_i(P, M)  \neq 0$ and so by our previous argument we have $\mu_{i+1}(\m, P) \neq 0$.

The ring $R = A/P$ is a one-dimensional domain. Let $x \in \m/P$ be $R$-regular. Set $I = (x)$. Let $e_1(I, -)$ be the modified multiplicity \wrt \ $I$ on the category of all finitely generated $R$-modules of dimension $\leq 1$. We note that $e_1(I, R) = \ell_A(A/(P, x))$.
We also have
\[
e_1(I, \Ext^i_A(A/P, M)) = \ell_{A_P}\left(\Ext^i_{A_P}(\kappa(P), M_P)\right) e_1(I, R) = \mu_i(P, M)e_1(I, R).
\]
We also have
$$ \ell\left(\Ext^{i+1}_A(A/(P,x), M)\right) \leq \ell(A/(P, x))\mu_{i+1}(\m, M) = \mu_{i+1}(\m, M)e_1(I, R). $$
The exact sequence $0 \rt A/P \xrightarrow{x} A/P \rt A/(P, x) \rt 0$ induces a long exact sequence in cohomology
\[
\Ext^i_A(A/P, M))  \xrightarrow{x} \Ext^i_A(A/P, M)) \xrightarrow{\delta} \Ext^{i+1}_A(A/(P,x), M).
\]
Let $L = \image \delta$. We have $e_1(I, \Ext^{i}_A(A/P, M)) \leq \ell(L)$, by a result of Serre cf., \cite[4.7.6]{BH}. We also have $\ell(L) \leq \ell\left(\Ext^{i+1}_A(A/(P,x), M)\right)  \leq  \mu_{i+1}(\m, M)e_1(I, R)$.
The result follows as $e_1(I,R) > 0$.
\end{proof}

\section{Proof of Theorem \ref{bass}}
In this section we prove Theorem \ref{bass}. We also give a construction which is crucial for us.

\begin{construction}
\label{b-const} Let $(A,\m)$ be a complete Noetherian local ring of dimension $d$ and let $M$ be a
finitely generated $A$-module
Let $\I \colon  0 \rt \I^0 \rt \I^1 \rt \cdots \rt \I^d\xrightarrow{f_d}  \I^{d+1} \rt \cdots$ be a minimal injective resolution of $M$. Let $r_i = \mu_i(\m, M)$ for $i \geq 0$. Let $E = E_A(A/\m)$ be the injective hull of the residue field of $A$.
 Let $\Gamma_\m(-)$ be the $\m$-torsion functor. Let $\Gamma_\m(\I) = \G$.
We have
$$ \G \colon 0 \rt E^{r_0} \rt \cdots E^{r_{d-1}} \xrightarrow {\partial_{d-1} } E^{r_d} \xrightarrow{\partial_d} E^{r_{d+1}} \rt \cdots.$$
Set $B = \image \partial_{d-1}$, $Z = \ker \partial_{d}$. We have a complex
$$\G^\prime \colon 0 \rt E^{r_0} \rt \cdots E^{r_{d-1}}  \rt B \rt 0.$$
Dualizing we obtain a complex
$$\mathcal{C}\colon  0 \rt B^\vee \rt F_{d-1} \xrightarrow{\delta_{d-1}} \cdots \rt F_0 \rt 0, \quad \text{where} \ F_i = A^{r_i}.$$
The homology of $\C$ is $H^i_\m(M)^\vee$ at degree $i$.
\end{construction}

\s\label{good} If $A$ is complete then $A$ is a homomorphic image of a Gorenstein local ring $R$ with $\dim R = \dim A = d$. Then it follows from \cite[3.5.11]{BH} that $\dim H^i_\m(M)^\vee \leq  i$.
\begin{corollary}
\label{bijlee}(with hypotheses as in \ref{b-const}). Let $P \in \Assh(A)$. Then $B^\vee_P$ is free $A_P$-module of rank $\sum_{i =0}^{d-1}(-1)^i\mu_{d-1-i}(\m, M)$.
\end{corollary}
\begin{proof}
By \ref{good} it follows that $\mathcal{C}_P$ is exact for all $P \in \Assh(A)$. The result follows.
\end{proof}

Next we give
\begin{proof}[Proof of Theorem \ref{bass}]
We may assume that $A$ is complete. We make the construction as in \ref{b-const}. We have
$\sum_{i =0}^{d-1}(-1)^i\mu_{d-1-i}(\m, M) = \rank_{A_P} B^\vee_P \geq 0$ for all primes $P \in \Assh(A)$.

We now prove the equivalent conditions.
Clearly (ii) and (iii) are equivalent and they imply (i). Assume (i). Then $B^\vee_P = 0$ for all primes $P \in \Assh(A)$. So $\dim B^\vee \leq d -1$. Set $D = \ker \delta_{d-1}$. Then $D/B^\vee = H^{d-1}_\m(M)^\vee$ which has dimension $\leq d-1$. So $\dim D \leq d-1$. Consider the complex which is a truncation of $\mathcal{C}$.
$$\mathcal{D}\colon  0  \rt F_{d-1} \xrightarrow{\delta_{d-1}}F_{d-2} \rt \cdots \rt F_0 \rt 0, \quad \text{where} \ F_i = A^{r_i}.$$
It has homology $H^i_\m(M)^\vee $ for $i \leq d-2$ and $D$ for $i = d-1$. Thus $\dim H_i(\mathcal{D}) \leq i$ for all $i$. As $\dim A = d$ it follows by Theorem \ref{int} that $\mathcal{D}$ is exact. So $H^i_\m(M)^\vee = 0$ for all $i \geq d-2$ and $D = 0$. As $H^{d-1}_\m(M)^\vee$ is a quotient of $D$ it follows that $H^{d-1}_\n(M)^\vee =0$. So $H^i_\m(M) = 0$ for all $i \leq d -1$. The result follows.
\end{proof}
We now give our proof of Bass conjecture.
\begin{theorem}\label{bc}
Let $(A,\m)$ be a Noetherian local ring. Assume there exists   a finitely generated $A$-module $M \neq 0$ with $\injdim M < \infty$. Then $A$ is \CM.
\end{theorem}
\begin{proof}
We may assume that $A$ is complete. Let $d = \dim A$ and $t = \depth A$. Suppose if possible $A$ is not \CM. Then $t  < d$. We also have $\injdim M = t < d$.  We make the construction as in \ref{b-const}. As $r_d = 0$ so $B = 0$. So $B^\vee = 0$. By \ref{bijlee} it follows that $\sum_{i =0}^{d-1}(-1)^i\mu_{d-1-i}(\m, M) = 0$. By Theorem \ref{bass} we get that $M$ is a maximal \CM \ $A$-module. In particular $\dim M = d$. But $\dim M \leq \injdim M$. So we have $d \leq t < d$ which is a contradiction. So $A$ is \CM.
\end{proof}
\section{Proof of Conjecture \ref{my-conj}}
We give
\begin{proof}[Proof of Conjecture \ref{my-conj}]
 When $d = 0$ the result follows from \ref{zero}. So assume $d > 0$.
We may assume that $A$ is complete.
  Let $P \in \Assh(A)$. By \ref{bella} we have $\mu_0(P, M) \leq 1$. As $M_P$ is a faithful $A_P$-module we get $\mu_0(P, M) \geq 1$. So $\mu_0(P, M) = 1$. By \ref{zero} again we get $M_P = \omega_P$. As $A$ is complete it is a quotient of a complete Gorenstein local ring $R$ with $\dim R = d$. Set $(-)^\dagger = \Hom_R(-, R)$. By local duality we have $M^\dagger = H^d_\m(M)^\vee$.
  So by \ref{rachel},  we have $e_d(\m, H^d_\m(M)^\vee) = e_d(\m, A)$. We make the construction as in \ref{b-const}. Let $\I$ be  a minimal injective resolution of $M$.
   Let $\Gamma_\m(-)$ be the $\m$-torsion functor. Let $\Gamma_\m(\I) = \G$.
We have
$$ \G \colon 0 \rt E^{r_0} \rt \cdots E^{r_{d-1}} \xrightarrow {\partial_{d-1} } E^{r_d} \xrightarrow{\partial_d} E^{r_{d+1}} \rt \cdots.$$
Set $B = \image \partial_{d-1}$, $Z = \ker \partial_{d}$.  By our assumption $r_d = 1$. We have two  exact sequences $0  \rt B \rt Z \rt H^d_\m(M) \rt 0$ and $0 \rt Z \rt E \rt U \rt 0$ where $U = \image \partial_d$.
Dualizing we get two exact sequences $0 \rt H^d_\m(M)^\vee \rt Z^\vee \rt B^\vee \rt 0$ and $0 \rt U^\vee \rt A \rt Z^\vee \rt 0$. By the first exact sequence we get that $e_d(\m, Z^\vee) \geq e_d(\m, H^d_\m(M)^\vee)  = e_d(\m, A)$. By the second exact sequence we get $e_d(\m, Z^\vee) \leq  e_d(\m, A)$. So $e_d(\m, Z^\vee) = e_d(\m, H^d_\m(M)^\vee)  = e_d(\m, A)$. Thus $e_d(\m, B^\vee) = e_d(\m, U^\vee) = 0$. It follows that $\dim B^\vee < d$. Let $P \in \Assh(A)$. Then $B^\vee_P = 0$. By \ref{bijlee},  $B^\vee_P$ is free $A_P$-module of rank $\sum_{i =0}^{d-1}(-1)^i\mu_{d-1-i}(\m, M)$. As $B^\vee_P = 0$ we obtain
$\sum_{i =0}^{d-1}(-1)^i\mu_{d-1-i}(\m, M) = 0$.  By Theorem \ref{bass} we get that $M$ is a MCM $A$-module.

 As $M$ is a MCM $A$-module it is a MCM $R$-module. So $M \cong (M^\dagger)^\dagger$ by \cite[3.3.10]{BH}. By \cite[1.9]{S},   if $P \in \Ass M$ then $P \in \Ass M^\dagger$ such that $\dim A/P = \dim M^\dagger$. But $\dim M^\dagger = d$. So if $P \in \Ass M$ we get that
$\dim A/P = d$. But $M$ is faithful. So by \ref{faithful},  $\Ass A \subseteq \Ass M$. Thus $A$ is unmixed. As $e_d(\m, U^\vee) = 0$ we get that $\dim U^\vee < d$. As $U^\vee$ is a submodule of $A$ and as $A$ is unmixed we get $U^\vee = 0$. So $U = 0$. It follows that $r_j = 0$ for $j > d$. Thus $\injdim M < \infty$. So $A$ is \CM \ by Bass's conjecture, see \ref{bc}.   As $M$ is a faithful
  \CM \ $A$-module of type one it is isomorphic to $\omega_A$.
\end{proof}
\section{Proof of Theorem \ref{rank-non-cm}}
In this section we give
\begin{proof}[Proof of Theorem \ref{rank-non-cm}]
Let $P \in \Assh(A)$. We have $\mu_0(P, M) \leq \mu_d(\m, M)$. As $M$ has rank $r$ we get $M_P = A_P^r$. So $\mu_0(P, M) = r \type(A_P)$. Thus $\mu_d(\m, M) \geq r$. If $\mu_d(\m, M) = r$ then $A_P$ is Gorenstein for all $P \in \Assh(A)$.

We now prove the equivalent conditions. The assertion (ii) $\implies$ (i) is trivial. We prove the converse. We note that we can assume that $A$ is complete, the only thing we need to show that if $\wh{A}$ is \CM \ and $\wh{A}_Q$ is Gorenstein for all minimal primes $Q$ then $A_P$ is Gorenstein for all minimal primes $P$ of $A$. To see this choose a minimal prime $Q$ of $\wh{A}$ minimal over $\wh{A}/P$ such that $\dim \wh{A}/Q = \dim A/P$. We have $Q$ is a minimal prime of $\wh{A}$ and we have a local flat map $A_P \rt \wh{A}_Q$. As $\wh{A}_Q$ is Gorenstein it follows from \cite[23.4]{M}, that $A_P$ is Gorenstein. With these arguments we assume that $A$ is complete.

As argued before, if $r = \mu_d(\m, M)$ then $A_P$ is Gorenstein for all primes $P \in \Assh(A)$. Thus we have nothing to show if $d = \dim A = 0$. So assume $d > 0$.
We make the construction as in \ref{b-const}. Let $\I$ be  a minimal injective resolution of $M$.
   Let $\Gamma_\m(-)$ be the $\m$-torsion functor. Let $\Gamma_\m(\I) = \G$.
We have
$$ \G \colon 0 \rt E^{r_0} \rt \cdots E^{r_{d-1}} \xrightarrow {\partial_{d-1} } E^{r_d} \xrightarrow{\partial_d} E^{r_{d+1}} \rt \cdots.$$
Set $B = \image \partial_{d-1}$, $Z = \ker \partial_{d}$.  By our assumption $r_d = r$. We have two  exact sequences $0  \rt B \rt Z \rt H^d_\m(M) \rt 0$ and $0 \rt Z \rt E^r \rt U \rt 0$ where $U = \image \partial_d$.
Dualizing we get two exact sequences $0 \rt H^d_\m(M)^\vee \rt Z^\vee \rt B^\vee \rt 0$ and $0 \rt U^\vee \rt A^r \rt Z^\vee \rt 0$.
As $A$ is complete it is a quotient of a complete Gorenstein local ring $R$ with $\dim R = d$. Set $(-)^\dagger = \Hom_R(-, R)$. We have $H^d_\m(M)^\vee = \Hom_R(M, R)$. By \ref{rachel} we obtain $ e_d(\m, H^d_\m(M)^\vee)  = e_d(\m, A^r).$
By the first exact sequence we get that $e_d(\m, Z^\vee) \geq e_d(\m, H^d_\m(M)^\vee)  = e_d(\m, A^r)$. By the second exact sequence we get $e_d(\m, Z^\vee) \leq  e_d(\m, A^r)$. So $e_d(\m, Z^\vee) = e_d(\m, H^d_\m(M)^\vee)  = e_d(\m, A^r)$. Thus $e_d(\m, B^\vee) = e_d(\m, U^\vee) = 0$. It follows that $\dim B^\vee < d$. Let $P \in \Assh(A)$. Then $B^\vee_P = 0$. By \ref{bijlee},  $B^\vee_P$ is free $A_P$-module of rank $\sum_{i =0}^{d-1}(-1)^i\mu_{d-1-i}(\m, M)$. As $B^\vee_P = 0$ we obtain
$\sum_{i =0}^{d-1}(-1)^i\mu_{d-1-i}(\m, M) = 0$.  By Theorem \ref{bass} we get that $M$ is a MCM $A$-module.

As $M$ is a MCM $A$-module it is a MCM $R$-module. So $M \cong (M^\dagger)^\dagger$ by \cite[3.3.10]{BH}. By \cite[1.9]{S} if $P \in \Ass M$ then $P \in \Ass M^\dagger$ such that $\dim A/P = \dim M^\dagger$. But $\dim M^\dagger = d$. So if $P \in \Ass M$ we get that
$\dim A/P = d$. But $M$ has rank $r$. Therefore $M_Q = A_Q^r$ for each prime $Q \in \Ass(A)$. So  $\Ass A \subseteq \Ass M$. Thus $A$ is unmixed. As $e_d(\m, U^\vee) = 0$ we get that $\dim U^\vee < d$. As $U^\vee$ is a submodule of $A^r$ and as $A$ is unmixed we get $U^\vee = 0$. So $U = 0$. It follows that $r_j = 0$ for $j > d$. Thus $\injdim M < \infty$. So $A$ is \CM \ by Bass's conjecture, see \ref{bc}.   As $M$ is maximal \CM \ $A$-module with finite injective dimension it follows that $M \cong \omega_A^s$ for some $s$. As $\mu_d(\m, M) = r$ it follows that $s = r$.
\end{proof}
\section{Proof of Theorem \ref{my-thm}}
In this section we give
\begin{proof}[Proof of Theorem \ref{my-thm}]
By
Theorem \ref{rank-non-cm} the result follows if $n  = d$. So assume $n > d$. Let $P \in \Assh(A)$, By \ref{bella} we get that $\mu_{n-d}(P, M) \leq 1$.
If $\mu_{n-d}(P, M) = 0$ then $\injdim M_P  = \injdim A_P^r < \infty$. So $A_P$ is Gorenstein. We assert that $\mu_{n-d}(P, M) = 1$ is not possible. Suppose if possible that $\mu_{n-d}(P, M) = 1$.
 Set $D = E_{A_P}(\kappa(P))$.
We have an exact sequence
$$ 0 \rt A_P^r \rt D^{s_0} \rt \cdots D^{s_{n-d-1}} \xrightarrow{f} D \xrightarrow{g} D^{s_{n-d + 1}} \rt \cdots. $$
Set $C = \image f$. Then $\ell_{A_P}(C_P) = t\ell_{A_P}(A_P)$ for some $t \geq 0$. It follows that $C = 0$ or $ \image(g) = 0$. So $\injdim A_P < \infty$ and so $A_P$ is Gorenstein.
This implies that  $\mu_{n-d}(P, M) = 0$. So our assertion holds. So $A_P$ is Gorenstein for all primes $P \in \Assh(A)$.

By our above argument the result holds when $d = 0$. So assume that $d > 0$. Note we can assume that $A$ is complete.
 The only thing we need to show that if $\wh{A}$ is \CM \ and $\wh{A}_Q$ is Gorenstein for all minimal primes $Q$ then $A_P$ is Gorenstein for all minimal primes $P$ of $A$. But this is already proved in proof of Theorem \ref{rank-non-cm}.

  Let $\Gamma_\m(-)$ be the $\m$-torsion functor. Let $\Gamma_\m(\I) = \G$.
We have
$$ \G \colon 0 \rt E^{r_0} \rt \cdots E^{r_{d-1}} \xrightarrow {\partial_{d-1} } E^{r_d} \xrightarrow{\partial_d} E^{r_{d+1}} \rt \cdots.$$
Set $B = \image \partial_{d-1}$, $Z = \ker \partial_{d}$.
 We have an  exact sequences $0  \rt B \rt Z \rt H^d_\m(M) \rt 0.$
Dualizing we get an exact sequences $0 \rt H^d_\m(M)^\vee \rt Z^\vee \rt B^\vee \rt 0$. Let $P \in \Assh(A)$.  By \ref{bijlee},  $B^\vee_P$ is free $A_P$-module of rank $\sum_{i =0}^{d-1}(-1)^i\mu_{d-1-i}(\m, M)$.
As $A$ is unmixed we get $\Assh(A) = \Ass(A) = \Min(A)$. Thus $B^\vee$ has a rank.

Claim: $H^d_\m(M)^\vee$ has rank $r$.

Proof of Claim: As $A$ is complete there exists a complete Gorenstein ring $R$ mapping onto $A$ with $\dim R = \dim A = d$. We have by local duality $H^d_\m(M)^\vee = \Hom_R(M, R)$. Let $P \in \Ass(A) = \Assh(A) = \Min(A)$. Let $Q$ be the pre-image of $P$ in $R$. We note that $Q$ is a minimal prime of $R$ and induces a surjection $R_Q \rt A_P$. We have
$$H^d(M)^\vee_P = \Hom_{R_Q}(M_P,R_Q) \cong \Hom_{R_Q}(A_P^r, R_Q) \cong A_P^r.$$
The last isomorphism holds as $A_P$ is Gorenstein. Thus $H^d_\m(M)^\vee$ has rank $r$.

It follows that $Z^\vee$ has a rank. We have an minimal exact sequence
$$ \cdots\rt F^{r_r} \rt \cdots \xrightarrow{f} A^{r_n} \xrightarrow{g} A^{r_{n-1}} \cdots \rt A^{r_d} \rt Z^\vee \rt 0.$$
Here by assumption $r_n \leq 1$. So we have exact sequence $0 \rt \image(f) \rt A^{r_n} \rt \image(g) \rt 0$. Both $\image(f)$ and $\image(g)$ have a rank  as $Z^\vee$ has a rank. So $\rank(\image f)$ or $\rank(\image g)$ is zero.
Both $\image (f)$ and $\image(g) $ are submodules of a free $A$-module. As $A$ is unmixed, it follows that $\image(f) $ or $\image(g)$ is zero.  It follows that $\mu_j(\m, M) = 0$ for $j > n$.
So  $\injdim M < \infty$. Thus $\mu_j(\m, M) = 0$ for $j > d$. By Bass conjecture $A$ is \CM, see \ref{bc}.
\end{proof}

We show that the assertion of the theorem need not hold if $M$ does not have a rank.
\begin{example}\label{alive}
  Let $(A,\m)$ be a Gorenstein Artin local ring which is not a field. Let $k = A/\m$. Let $\Syz_n(k)$ denote the $n^{th}$ syzygy of $k$. Then $\mu_n(\m, \Syz_n(k)) = 1$, but $\Syz_n(k)$ does not have finite injective dimension.
\end{example}

Next we give an example which shows that $M$ need not be \CM \ even if it satisfies the assumption of the theorem.
\begin{example}\label{sty}
  Let $(A,\m)$ be a one-dimensional Gorenstein  local ring. Let $x \in \m$ be a non-zero divisor of $A$. Set $M = A \oplus A/(x)$, Then $\rank M = 1$. Also $\mu_2(\m, M) = 0$. But $M$ is not \CM.
\end{example}
We now give
\begin{proof}[Proof of Proposition \ref{nuv}]
We note that $A$ is \CM \ and $\injdim M < \infty$ by Theorem \ref{my-thm}. As $A$ is a quotient of an Gorenstein local ring,  $A$ has a canonical module $\omega$. Since $\injdim M$ is finite, by an exercise problem \cite[3.3.28]{BH}, $M$ has a finite minimal $\omega$ resolution, i.e., an exact sequence
\[
0 \rt \omega^{s_p} \xrightarrow{\psi_p} \omega^{s_{p-1}}  \rt \cdots \rt \omega^{s_1} \xrightarrow{\psi_1} \omega^{s_0} \rt M \rt 0,
\]
where $\image \psi_i \subseteq \m \omega^{s_{i-1}}$ for all $i$ and $s_i = \mu_{d-i}(\m, M)$ for all $i$. It follows that $\nu(M) = s_0\type(M) \geq r \type(M)$ by Theorem \ref{rank-non-cm}. Note equality holds if and only if $\mu_d(\m, M) = r = \rank M$. Then again by Theorem \ref{rank-non-cm} we get that $M$ is maximal \CM. As $\injdim M < \infty$ it follows from the same excercise problem \cite[3.3.28]{BH} we get $M \cong \omega^s$ for some $s$. As $\mu_d(\m, M) = r$ we get $s = r$. The result follows.
\end{proof}

\end{document}